\documentclass[12pt]{article}
\usepackage{cite}
\usepackage{amsmath,amsthm,amssymb}
\usepackage{hyperref}
\usepackage{latexsym}
\usepackage{stmaryrd}
\usepackage[sc,osf]{mathpazo}

\def\bct{\begin{center}}
\def\ect{\end{center}}
\def\beg{\begin}
\def\bit{\begin{itemize}}
\def\eit{\end{itemize}}

\def\<{\langle}
\def\>{\rangle}

\def\mbb{\mathbb}

\def\tn{\textnormal}

\newtheorem{thm}{Theorem}[section]
\newtheorem{lem}[thm]{Lemma}

\newtheorem{cor}[thm]{Corollary}

\title{Conformally K\"ahler Surfaces and Orthogonal Holomorphic Bisectional Curvature}
\author{Mustafa Kalafat \and Caner Koca}

\begin{document}
\maketitle
\begin{abstract} We show that a compact complex surface which admits a conformally K\"ahler metric $g$ of positive orthogonal holomorphic bisectional curvature is biholomorphic to the complex projective plane. In addition, if $g$ is a Hermitian metric which is Einstein, then the biholomorphism can be chosen to be an isometry via which $g$ becomes a multiple of the Fubini-Study metric.
  \end{abstract}

\section{Introduction}

Let $(M,J)$ be a complex manifold. A Riemannian metric $g$ on $M$ is called a {\em Hermitian metric}
if the complex structure $J:TM\rightarrow TM$ is an orthogonal transformation at every point on $M$ with respect to the metric $g$, that is,
 $g(X,Y)=g(JX,JY)$ for tangent vectors $X,Y\in T_p M$ for all $p\in M$. In this case, the triple $(M,J,g)$ is called a {\em Hermitian manifold}.
The compatibility between the complex structure and the metrics allows one to define further notions of curvature:

\begin{enumerate}
\item  The {\em holomorphic sectional curvature} in the direction of a unit tangent vector $X$ is defined by
$$\tn{H}(X):=\tn{Rm}(X,JX,X,JX) = g(R(X,JX)JX,X)$$
Thus, $\tn{H}(X)$ is the {\em sectional curvature} of the complex line (i.e. the $J$-invariant real plane) spanned by the real basis $\{U,JU\}$.  One can also think of this as the restriction of the sectional curvature function on the Grassmannian of real 2-planes $G_2(T_p M)$ to the Grassmannian of complex lines $G_1^{\mbb C} (T_p M)$, i.e. the $\mbb{CP}_{n-1}$-many complex lines at a point:
$$\tn{H}: G_1^{\mbb C}(T_p M) \longrightarrow \mbb R.$$
As an example, the complex projective spaces $\mbb{CP}_n$ with Fubini-Study metric are of constant positive holomorphic curvature $4$. In general, the sectional curvatures lie in the interval $[1,4]$ and a plane has sectional curvature 4 precisely when it is a complex line.

\item The \emph{holomorphic bisectional curvature} associated to a given pair of unit vectors $X,Y\in T_p M$ is defined as
$${\tn{H}(X,Y)}:=\tn{Rm}(X,JX,Y,JY) = g(R(X,JX)JY,Y).$$
Note that the holomorphic sectional curvature is determined by the holomorphic bisectional curvature as
\begin{equation}\label{holsecvsbisec}
\tn{H}(X)=\tn{H}(X,X).\end{equation}
\end{enumerate}

Many results are known about holomorphic sectional and bisectional curvatures in the literature when the Hermitian metric $g$ is \emph{K\"ahler}, i.e. when $\nabla J=0$. K\"ahler metrics have holonomy in $U(n)$, which allows us to write the bisectional curvature as a sum of two \emph{sectional} curvatures as follows:
\begin{equation}\label{bisecexpansion}
{\tn{H}(X,Y)} = {\tn{K}(X,Y)}+{\tn{K}(X,JY)}
\end{equation}
where $\tn{K}(\cdot,\cdot)$ stands for the sectional curvatures (hence the name ``bisectional").

The Uniformization Theorem \cite{unifhawley53,unifigusa54} asserts that a complete K\" ahler manifold of constant holomorphic sectional curvature $\lambda$ is necessarily a quotient of one of the following three models,
\begin{enumerate}
\item $\lambda>0$ ~$\implies$~ $\mbb{CP}_n$~~ endowed with Fubini-Study metric.
\item $\lambda=0$ ~$\implies$~ $\mbb C ^n$~~ endowed with flat metric.
\item $\lambda<0$ ~$\implies$~ $\mbb D^n \subset \mbb C^n$~~ endowed with hyperbolic metric.
\end{enumerate}
In particular such a metric has to be Einstein, i.e. of constant Ricci curvature. Moreover, a complete K\" ahler manifold of positive constant holomorphic sectional curvature is simply connected by \cite{Synge:1936,kobayashi}. This means that there is no quotient in the first case.

Once the constant curvature case has been settled, one immediately wonders  what can be said about the non-constant case. In this direction, we know that the positivity of bisectional curvature uniquely determines the biholomorphism type:
\begin{thm}[Frankel Conjecture, Siu-Yau Theorem \cite{SY:1980}] \label{bisecfrankeln}
Every compact K\" ahler manifold of positive bisectional holomorphic curvature is biholomorphic to the complex projective space.
\end{thm}
This theorem does not, however, specify the metric in question. Nevertheless, if we in addition assume that the metric is \emph{Einstein}, then the metric is unique, too:
\begin{thm}[\cite{berger65,goldbergkobayashi67}]  \label{bgk}
An $n$-dimensional compact connected K\"ahler manifold with an Einstein (or constant scalar curvature) metric of positive holomorphic bisectional curvature is globally isometric to $\mbb{CP}_n$ with the Fubini-Study metric up to rescaling.
\end{thm}

In this paper we will show in 4-dimensions that the hypothesis of these theorems can be relaxed. Our main theorem is as follows:

\vspace{.05in}

\noindent {\bf Theorem \ref{finalthm}.}
{\em If a compact complex surface $(M,J)$ admits a \emph{conformally K\"ahler} metric of {\em positive orthogonal holomorphic bisectional curvature}, then it is biholomorphically equivalent to the complex projective plane $\mbb{CP}_2$.}

\vspace{.05in}

Here, a Hermitian metric $g$ is said to have \emph{positive orthogonal holomorphic bisectional curvature} if the bisectional holomorphic curvature of any two complex lines in every tangent space are positive, i.e. $$H(X,Y)>0$$ whenever $X\perp Y $ and $JY$. Thus, for example, if the metric $g$ is \emph{K\"ahler}, equations \eqref{holsecvsbisec} and \eqref{bisecexpansion} imply that positivity of sectional curvature implies that of holomorphic bisectional curvature, which in turn implies positivity of holomorphic sectional curvature. The positivity of \emph{orthogonal} holomorphic bisectional curvature is in some sense a complementary condition to the latter. None of the converse directions is in general true. On the other hand, if the metric is not K\"ahler, but merely Hermitian, then the holonomy of the metric is not in $U(2)$, and hence the equation \eqref{bisecexpansion} does not hold; therefore, positivity of the sectional curvature does not imply the positivity of the bisectional curvature.

A metric $g$ is called \emph{conformally K\"ahler} if the conformally rescaled metric $ug$ is K\"ahler, where $u$ is a positive smooth function. On a compact manifold, if $g$ is K\"ahler to begin with, $ug$ will most definitely not be K\"ahler (with respect to the same complex structure
\footnote{Here one should insist on being K\" ahler with respect to the {\em same} complex 
structure. As an example,
the Page metric on  $\mbb{CP}_2\sharp\, \overline{\mbb{CP}}_2$ has an orientation-reversing 
conformal map, consequently has two different
conformal rescalings which are K\"ahler but for different complex structures. 
Interested reader may consult to \cite{apostolovetalambi} and related papers for more regarding this 
situation.} 
) unless $u$ is identically constant. Our theorem has the following nice application.

\vspace{.05in}

\noindent {\bf Corollary \ref{mainthmcor}.}
{\em If $(M,J,g)$ is a Hermitian manifold and $g$ is an Einstein metric of positive orthogonal bisectional curvature. Then there is a biholomorphism between $M$ and $\mbb{CP}_2$ via which $g$ becomes the Fubini-Study metric up to rescaling.}

\vspace{.05in}

We remark in passing that the 4-dimensional version of the Frankel Conjecture had already been proven by Andreotti \cite{andreotti} in 1957. Also, in \cite{kocathesis,kocapage}, the second author obtained a result analogous to Corollary \ref{mainthmcor} with the positivity assumption on the \emph{sectional} curvature.

\noindent{\bf Acknowledgements.} The authors would like to thank Claude LeBrun for
suggesting us this problem, and for his help and encouragement. The authors also thank F. Belgun, 
Peng Wu for useful inquiries, and the referee for useful remarks which greatly improved the original manuscript. This work is partially supported by the grant 
 $\sharp$113F159 of T\"ubitak\footnote{Turkish science and research council.}.

\newpage 

\section{Bisectional curvature and Weitzenb\" ock formula}\label{secestimates}

In order to better understand the bisectional curvature of Hermitian metrics which are not K\"ahler, it is crucial to describe a way of computing this curvature using $2$-forms. Let $X,Y$ be two unit tangent vectors on a Hermitian $4$-manifold $(M,J,g)$ . Then
\begin{align*}
\tn{H}(X,Y) 
            & = \langle \mathcal R(X\wedge JX), Y\wedge JY\rangle
\end{align*}
where $\mathcal R:\Lambda^2 TM \rightarrow \Lambda^2 TM$ is the curvature operator acting on 2-vectors. Note that $X\wedge JX$ and $Y\wedge JY$ correspond to two complex lines generated by $\{X,JX\}$ and $\{Y,JY\}$. 
We next describe a useful interpretation of the planes in the tangent space via 2-forms. 
Sectional curvatures at a point $p\in M$ can be thought as a function on the Grassmannian of oriented two planes in the corresponding tangent space:
$$\tn{K} : G_2^+(T_p M) \longrightarrow \mbb R.$$

In dimension $4$ it is well-known that one can describe this Grassmannian in terms of 2-forms as \cite{jeffdg}
 $$G_2^+(T_p M)\approx
 \{(\alpha,\beta)\in (\Lambda^2_+\oplus\Lambda^2_-)_p : |\alpha|=|\beta|=1/\sqrt{2}\}\approx S^2\times S^2$$
where $\Lambda^2_{\pm}$ stands for the self-dual and anti-self-dual $2$-forms, i.e. $\Lambda^2_{\pm} = \{\phi\in\Lambda^2 : *\phi = \pm \phi \}$, where $*$ is the Hodge-$*$ operator determined by the metric $g$. 
More explicitly, given an oriented plane in $T_p M$, one can choose a unique orthogonal basis $\{e_1,e_2\}$, such that the wedge of their duals (with respect to $g$) $e^1\wedge e^2\in\Lambda^2$ is decomposable uniquely in such a way that 
~$e^1\wedge e^2= \alpha+\beta$ where $\alpha\in\Lambda^2_+,~ \beta\in \Lambda^2_-,~|\alpha|=|\beta|=1/\sqrt{2}$.
Conversely, starting with a decomposable 2-form $\omega\in\Lambda^2$ i.e. $\omega=\theta\wedge\delta$ for some $\theta,\delta\in \Lambda^1$,
the duals  $\{\theta^\sharp,\delta^\sharp\}$ give an oriented basis for a plane.

In this correspondence, a complex line in $T_p M$ corresponds to a $2$-form of the form
${\omega\over 2} + \varphi$  for an anti-self-dual 2-form $\varphi\in(\Lambda^2_-)_p$ with $|\varphi|=1/\sqrt{2}$, where $\omega\in\Lambda^2_+$ is the associated $(1,1)$-form of the Hermitian metric $g$, i.e. $\omega (\cdot,\cdot)=g(J\cdot,\cdot)$. Notice that $|\omega| = \sqrt 2$.
With this correspondence in mind, we see that, given two unit tangent vectors $X,Y\in T_p M$ we can express the holomorphic bisectional curvature as
$$\tn{H}(X,Y) = \langle \mathcal R \left({\omega\over 2} + \varphi \right), {\omega\over 2} + \psi\rangle$$
where the anti-self-dual forms $\varphi,\psi\in(\Lambda^2_-)_p$ of length $1/\sqrt 2$ correspond to the complex lines spanned by the real bases $\{X,JX\}$ and $\{Y,JY\}$, respectively. Here, $\mathcal R:\Lambda^2 \rightarrow \Lambda^2$ is the curvature operator on $2$-forms.

Let us recall the decomposition of the curvature operator $\mathcal R :\Lambda^2\rightarrow \Lambda^2$. If $(M,g)$ is any oriented 4-manifold, then the decomposition $\Lambda^2 =\Lambda^2_+ \oplus \Lambda^2_-$ implies that the curvature operator $\mathcal R$ can be decomposed as
\begin{equation}
\label{curv}
{\mathcal R}=
\left(
\mbox{
\begin{tabular}{c|c}
&\\
$W_++\frac{s}{12}$&$\mathring{r}$\\ &\\
\cline{1-2}&\\
$\mathring{r}$ & $W_-+\frac{s}{12}$\\&\\
\end{tabular}
} \right)
\end{equation}
where $W_{\pm}$ is the self-dual/anti-self-dual Weyl curvature tensor, $s$ is the scalar curvature and $\mathring r$ stands for the tracefree part of the Ricci curvature tensor $r$ acting on $2$-forms. Note that $\mathcal R$ is a self-adjoint operator.

\begin{lem}\label{lemma.conf}
Let $(M,J,g)$ be a Hermitian manifold. If $g$ has positive orthogonal holomorphic bisectional curvature, then so does any other conformally equivalent metric $ug$.
\end{lem}
\begin{proof}
Let $\omega$ be the associated $(1,1)$-form of $g$. For any $\varphi\in\Lambda_-$ with $|\varphi|_g=1/\sqrt 2$, the 2-forms $\frac{\omega}{2}+\varphi$ and $\frac{\omega}{2}-\varphi$ are orthogonal with respect two the metric; and hence, they represent \emph{orthogonal} complex lines in the tangent space. Therefore, a Hermitian metric $g$ has positive orthogonal holomorphic bisectional curvature if and only if
 $$\langle \mathcal R \left({\omega\over 2} + \varphi \right), {\omega\over 2} - \varphi\rangle>0$$
for any $\varphi\in \Lambda_-$ with (pointwise) norm $|\varphi|_g=1/\sqrt{2}$. On the other hand, using the decomposition of the curvature tensor, and the fact that $\mathcal R$ is a \emph{self-adjoint} operator we see that
\begin{equation}\label{pohbc.cond}
\langle \mathcal R \left({\omega\over 2} + \varphi \right), {\omega\over 2} - \varphi\rangle = \langle  W_+ \left({\omega\over 2}\right), {\omega\over 2} \rangle - \langle  W_- \left({\varphi}\right), {\varphi} \rangle >0
\end{equation}
However, this condition is indeed \emph{conformally invariant}. More precisely, for a conformally related metric $\tilde g = ug$ we have
$$\langle\tilde W_+ \left(\frac{\tilde\omega}{2}\right), \frac{\tilde\omega}{2}\rangle_{\tilde g} - \langle\tilde W_-(\tilde\varphi),\tilde\varphi\rangle_{\tilde g} = \frac{1}{u}\left[ \langle  W_+ \left({\omega\over 2}\right), {\omega\over 2} \rangle_g - \langle  W_- \left({\varphi}\right), {\varphi} \rangle_g \right]$$
where $\tilde\omega = u\omega$ and $\tilde \varphi = u\varphi$.
K\"ahler form of the new metric and the representative $2$-form of the plane are different from the originals. They both rescale to keep their norms fixed.
Thus, by the same argument, $\tilde g$ has positive orthogonal holomorphic bisectional curvature.
\end{proof}

The following formula will be used:

\beg{thm}[Weitzenb\" ock Formula \cite{bourguignon}]
On a Riemannian manifold, the Hodge/modern Laplacian can be
expressed in terms of the connection/rough Laplacian as
\begin{equation}\label{weitz}(d+d^*)^2=\nabla^*\nabla-2W+{s\over 3} \end{equation}
where $\nabla$ is the Riemannian connection and $W$ is the Weyl
curvature tensor. \end{thm}

\begin{lem}\label{lemma.pos}
Let $(M,J)$ be a compact complex surface equipped with a conformally K\"ahler metric $g$ which has positive orthogonal holomorphic bisectional curvature. Then the intersection form on the second cohomology of $M$ is positive definite.
\end{lem}
In other words, the Betti number $b^2_-$, which stands for the number of negative eigenvalues of the cup product on $H^2(M,\mathbb R)$, is zero. Recall that, by Hodge theory, $b^2_-$ is the number is equal to the dimension of anti-self-dual \emph{harmonic} $2$-forms on the Riemannian manifold $(M,g)$.
\begin{proof}
Without loss of generality, we may assume that the conformally K\"ahler metric $g$ is indeed \emph{K\"ahler} with K\"ahler form $\omega$, by Lemma \ref{lemma.conf}. For K\"ahler metrics, $\omega$ is an eigenvector of $W_+$ with eigenvalue $\frac{s}{3}$, that is,
\begin{equation}\label{evalue}
\langle W_+(\omega),\omega\rangle = \frac{s}{6}|\omega|^2 = \frac{s}{3},
\end{equation}
 by \eqref{weitz} applied to $\omega$, or see \cite{Derdzinski}.

Now, let $\varphi$ any a be harmonic anti-self-dual 2-form. We claim that $\varphi$ is identically zero. If not, then the ``non-vanishing" set $V:=\{\varphi\neq 0\}$ is non-empty. If we apply the Weitzenb\"ock formula with the harmonic 2-form $\frac{\omega}{2}+\varphi$, we get
\begin{align*}
0&= \nabla^* \nabla (\frac{\omega}{2} +\varphi) - 2W(\frac{\omega}{2} +\varphi) + \frac{s}{3} (\frac{\omega}{2} +\varphi).
\end{align*}
Now, using the fact that $\nabla \omega = 0$ (K\"ahler condition) and taking the $L_2$-inner product with $\frac{\omega}{2}-\varphi$ we obtain
\begin{align*}
\hspace{-13mm}0&= -\int_M |\nabla \varphi|^2 d\mu_g - 2\int_M \left. \langle W_+ \left(\frac{\omega}{2}\right) ,\frac{\omega}{2} \rangle - \langle W_- (\varphi),\varphi\rangle \right. d\mu_g + \int_M \frac{s}{3} \left(\frac{|\omega|^2}{4} -|\varphi|^2\right) d\mu_g
\end{align*}

Now using \eqref{evalue} in the second integral above yields
\begin{align}
\label{int.step}0&= -\int_M |\nabla \varphi|^2 d\mu_g + 2 \int_M \left(   \langle W_- (\varphi),\varphi\rangle  - \frac{s}{6} |\varphi|^2 \right) d\mu_g.
\end{align}
We claim that the second integral is negative if $V$ is non-empty.

First of all, the integrand $\langle W_- (\varphi),\varphi\rangle - \frac{s}{6} |\varphi|^2 \leq 0$ on all of $M$: Outside $V$, $\varphi$ is zero; therefore, the inequality is trivially satisfied. On the other hand, on $V$, $\varphi$ is non-zero, thus at every point $p\in V$ we can consider the rescaled form $\tilde\varphi=\frac{1}{\sqrt 2} \frac{\varphi}{|\varphi|_p}$ which has the desired norm $1/\sqrt 2$. Since $g$ has positive orthogonal holomorphic bisectional curvature, by \eqref{pohbc.cond} applied to $\tilde\varphi$, we have
$$ \langle W_- (\varphi),\varphi\rangle = 2|\varphi|^2 \langle W_- (\tilde\varphi),\tilde\varphi\rangle < 2 |\varphi|^2 \langle W_+ \left(\frac{\omega}{2}\right),\frac{\omega}{2}\rangle = \frac{s}{6} |\varphi|^2.$$
Hence, on a small ball around $p$, the integral of  $ \langle W_- (\varphi),\varphi\rangle  -  \frac{s}{6} |\varphi|^2$ is strictly negative. This implies that $\int_M  \langle W_- (\varphi),\varphi\rangle  -  \frac{s}{6} |\varphi|^2 d\mu_g<0$, and thus, from \eqref{int.step}, we arrive at
\begin{align*}
\int_M |\nabla \varphi|^2 d\mu_g<0
\end{align*}
which is absurd. This contradiction shows that $V$ must be empty, and therefore $\varphi$ is identically zero, and hence $b^2_- =0$. Since we are on a compact manifold of K\"ahler type, we have $b^2_+ \geq 1$. Thus, the intersection form is positive definite.
\end{proof}

\newpage

\section{Main Result and a Corollary}
In this section we will prove our main result:

\begin{thm} \label{finalthm}
If a compact complex surface $(M,J)$ admits a \emph{conformally K\"ahler} metric $g$ of \emph{positive orthogonal holomorphic bisectional curvature}, then it is biholomorphically equivalent to the complex projective plane $\mbb{CP}_2$.
\end{thm}
\begin{proof}
As before, without loss of generality we may assume that $g$ is \emph{K\"ahler}, by Lemma \ref{lemma.conf}. Moreover, by Lemma \ref{lemma.pos}, $b^2_-=0$. In particular, $(M,J)$ is a \emph{minimal} complex surface, namely it cannot have any rational curve of self intersection $-1$.

Moreover, all plurigenera of $(M,J)$ are zero by Yau's Plurigenera Vanishing Theorem \cite{Yau:1974}. To see this, first note that, if $\varphi\in\Gamma(K^m)$ ($m\geq 0$) is a non-zero holomorphic section of $m$-th power of the canonical bundle $K$ of $M$, then its zero locus has area
\begin{equation}\label{canbun}
-m c_1 \cdot [\omega]>0.
\end{equation}
On the other hand, when the K\"ahler metric has positive orthogonal holomorphic bisectional curvature, \emph{its scalar curvature must be positive}, because at every point $p\in M$, we have by \eqref{pohbc.cond} that
$$\frac{s}{12} = \langle W_+ \left(\frac{\omega}{2}\right), \frac{\omega}{2} \rangle > \langle W_ - (\varphi),\varphi \rangle >0$$
for some non-zero $\varphi$ in the eigenspace, for some positive eigenvalue of $W_-$. Note that such an eigenvalue exists, because $W_-$ is trace-free.

Since $s>0$ on $M$, the the total scalar curvature, i.e. $\int_M s\,d\mu$, must be positive. However, $\int_M s\,d\mu = 4\pi c_1\cdot [\omega]>0$, and this contradicts to \eqref{canbun}.

Since all plurigenera vanish, the Kodaira dimension of $(M,J)$ must be $-\infty$. Moreover, since the intersection form is positive definite, $M$ cannot contain a self intersection $(-1)$-curve, and consequently the surface must be minimal. Now, by the Enriques-Kodaira classification of compact complex surfaces \cite{BPV}, we see that the $(M,J)$ must be a minimal rational or a ruled surface. However, since the signature $\tau$ of $M$ is positive (as $b^2_- = 0$), $(M,J)$ cannot be a ruled surface, and therefore it must be the complex projective plane $\mbb{CP}_2$, as required.
\end{proof}

We would like to conclude the paper with a nice application of our main theorem. Let $(M,J,g)$ be a compact Hermitian manifold of real dimension 4. Moreover, assume that $g$ is Einstein. In this case, LeBrun \cite[Proposition 1]{LeB:1995} showed that such a metric must be conformally K\" ahler. However, if we in addition assume that $g$ has positive orthogonal holomorphic bisectional curvature, $(M,J)$ has to be biholomorphically equivalent to $\mathbb{CP}_2$; and the conformally K\"ahler metric must indeed be a K\"ahler-Einstein metric to begin with \cite[Theorem 1]{LeB:1995}. However, by Matsushima--Lichnerowicz Theorem, the identity component of the isometry group of $g$ must be the maximal compact subgroup of the identity component of the automorphism group of $\mathbb{CP}_2$, in general for constant scalar curvature K\" ahler spaces \cite{matsushima,lichnerowicz}. It follows that $g$ is indeed $SU(3)$-invariant. But, $SU(3)$ acts transitively on $\mathbb{CP}_2=SU(3)/U(2)$, and therefore, up to scale, $g$ must be the Fubini-Study metric, which is the unique $SU(3)$-invariant metric on $\mathbb{CP}_2$.  This proves our corollary.

\begin{cor}\label{mainthmcor}
If $(M,J,g)$ is a Hermitian manifold and $g$ is an Einstein metric of positive orthogonal holomorphic bisectional curvature. Then there is a biholomorphism between $M$ and $\mbb{CP}_2$ via which $g$ becomes the Fubini-Study metric up to rescaling.
\end{cor}

\vspace{.05in}



{\small
\beg{flushleft}
\textsc{Tuncel\'i \" Un\' ivers\'ites\'i, Mekatron\'{\i}k M\" uhend\' isl\' i\v g\' i, 
Aktuluk, 62000, T\"urk\'iye.}\\
\textit{E-mail address :} \texttt{\textbf{mkalafat@tunceli.edu.tr}}
\end{flushleft}
}

{\small
\beg{flushleft}
\textsc{
Vanderbilt University, Department of Mathematics, Nashville, TN 37240. 
}\\
\textit{E-mail address :} \texttt{\textbf{caner.koca@vanderbilt.edu}}

\end{flushleft}
}


\newpage

\bibliography{confk}{}

\begin{thebibliography}{BPVdV84}

\bibitem[ACG]{apostolovetalambi}
V.~Apostolov, D.~M.~J. Calderbank, and P.~Gauduchon.
\newblock Ambitoric geometry {I}: Einstein metrics and extremal ambikaehler
  structures.
\newblock {\em To appear in J. Reine Angew. Math.}

\bibitem[And57]{andreotti}
Aldo Andreotti.
\newblock On the complex structures of a class of simply-connected manifolds.
\newblock In {\em Algebraic geometry and topology. {A} symposium in honor of
  {S}. {L}efschetz}, pages 53--77. Princeton University Press, Princeton, N.
  J., 1957.

\bibitem[Ber65]{berger65}
Marcel Berger.
\newblock Sur quleques vari\'et\'es riemanniennes compactes d'{E}instein.
\newblock {\em C. R. Acad. Sci. Paris}, 260:1554--1557, 1965.

\bibitem[Bou81]{bourguignon}
Jean-Pierre Bourguignon.
\newblock Les vari\'et\'es de dimension {$4$} \`a signature non nulle dont la
  courbure est harmonique sont d'{E}instein.
\newblock {\em Invent. Math.}, 63(2):263--286, 1981.
\newblock \href {http://dx.doi.org/10.1007/BF01393878}
  {\path{doi:10.1007/BF01393878}}.

\bibitem[BPVdV84]{BPV}
W.~Barth, C.~Peters, and A.~Van~de Ven.
\newblock {\em Compact complex surfaces}, volume~4 of {\em Ergebnisse der
  Mathematik und ihrer Grenzgebiete (3) [Results in Mathematics and Related
  Areas (3)]}.
\newblock Springer-Verlag, Berlin, 1984.

\bibitem[Der83]{Derdzinski}
Andrzej Derdzi{\'n}ski.
\newblock Self-dual {K}\"ahler manifolds and {E}instein manifolds of dimension
  four.
\newblock {\em Compositio Math.}, 49(3):405--433, 1983.

\bibitem[GK67]{goldbergkobayashi67}
Samuel~I. Goldberg and Shoshichi Kobayashi.
\newblock Holomorphic bisectional curvature.
\newblock {\em J. Differential Geometry}, 1:225--233, 1967.

\bibitem[Haw53]{unifhawley53}
N.~S. Hawley.
\newblock Constant holomorphic curvature.
\newblock {\em Canadian J. Math.}, 5:53--56, 1953.

\bibitem[Igu54]{unifigusa54}
Jun-ichi Igusa.
\newblock On the structure of a certain class of {K}aehler varieties.
\newblock {\em Amer. J. Math.}, 76:669--678, 1954.

\bibitem[Kob61]{kobayashi}
Shoshichi Kobayashi.
\newblock On compact {K}\"ahler manifolds with positive definite {R}icci
  tensor.
\newblock {\em Ann. of Math. (2)}, 74:570--574, 1961.

\bibitem[Koc12]{kocathesis}
Caner Koca.
\newblock {\em On {C}onformal {G}eometry of {K}ahler {S}urfaces}.
\newblock ProQuest LLC, Ann Arbor, MI, 2012.
\newblock Thesis (Ph.D.)--State University of New York at Stony Brook.

\bibitem[Koc14]{kocapage}
Caner Koca.
\newblock Einstein {H}ermitian metrics of positive sectional curvature.
\newblock {\em Proc. Amer. Math. Soc.}, 142(6):2119--2122, 2014.

\bibitem[LeB97]{LeB:1995}
Claude LeBrun.
\newblock Einstein metrics on complex surfaces.
\newblock In {\em Geometry and physics ({A}arhus, 1995)}, volume 184 of {\em
  Lecture Notes in Pure and Appl. Math.}, pages 167--176. Dekker, New York,
  1997.

\bibitem[Lic57]{lichnerowicz}
Andr{\'e} Lichnerowicz.
\newblock Sur les transformations analytiques des vari\'et\'es k\"ahl\'eriennes
  compactes.
\newblock {\em C. R. Acad. Sci. Paris}, 244:3011--3013, 1957.

\bibitem[Mat57]{matsushima}
Yoz{\^o} Matsushima.
\newblock Sur la structure du groupe d'hom\'eomorphismes analytiques d'une
  certaine vari\'et\'e k\"ahl\'erienne.
\newblock {\em Nagoya Math. J.}, 11:145--150, 1957.

\bibitem[SY80]{SY:1980}
Yum~Tong Siu and Shing~Tung Yau.
\newblock Compact {K}\"ahler manifolds of positive bisectional curvature.
\newblock {\em Invent. Math.}, 59(2):189--204, 1980.
\newblock \href {http://dx.doi.org/10.1007/BF01390043}
  {\path{doi:10.1007/BF01390043}}.

\bibitem[Syn36]{Synge:1936}
J.L. Synge.
\newblock On the connectivity of spaces of positive curvature.
\newblock {\em Quart. J. Math. (Oxford series)}, 7(1):316--320, 1936.

\bibitem[Via11]{jeffdg}
Jeff Viaclovsky.
\newblock Lecture notes: Topics in riemannian geometry, 2011.
\newblock URL: \url{http://www.math.wisc.edu/~jeffv/courses/865_Fall_2011.pdf}.

\bibitem[Yau74]{Yau:1974}
Shing~Tung Yau.
\newblock On the curvature of compact {H}ermitian manifolds.
\newblock {\em Invent. Math.}, 25:213--239, 1974.

\end{thebibliography}
\bibliographystyle{alphaurl}

\end{document}